\documentclass[reqno]{amsart}
\usepackage[b5paper]{geometry}

\usepackage{amsmath,amscd}
\usepackage{amssymb}
\usepackage{amsfonts}
\usepackage{newlfont}
\usepackage{graphicx}
\usepackage{color}
\usepackage{epstopdf}

\graphicspath{{./Figures/}}

\providecommand\newdefinition\newtheorem

\numberwithin{equation}{section}

\newtheorem{theorem}{Theorem}[section]
\newtheorem{lemma}[theorem]{Lemma}

\newtheorem{proposition}[theorem]{Proposition}
\newdefinition{definition}[theorem]{Definition}
\newdefinition{example}[theorem]{Example}
\newdefinition{construction}[theorem]{Construction}
\newdefinition{remark}[theorem]{Remark}

  \newcommand{\RSet}{{\mathbb{R}}}
  \newcommand{\ZSet}{{\mathbb{Z}}}
  \newcommand{\NSet}{{\mathbb{N}}}
  \newcommand{\QSet}{{\mathbb{Q}}}

\newcommand{\Sing}{\mathop{\mathrm{Sing}}\nolimits}

\newcommand{\rk}{\mathop{\mathrm{rk}}\nolimits}
\newcommand{\crk}{b_1'}
\newcommand{\HG}{H_1(G)}

\newcommand{\im}{\mathop{\mathrm{im}}\nolimits}

\newcommand{\T}{\mathop{\mathrm{\kern0pt T}}\nolimits} %
\newcommand{\codim}{\mathop{\mathrm{codim}}\nolimits}

\newcommand*{\F}[1][\w]{{\mathcal{F}}_{\!#1}}

\newcommand{\al}{\alpha}
\newcommand{\g}{\gamma}

\newcommand{\w}{\omega}

\renewcommand{\phi}{\varphi}

\renewcommand{\b}{{b_1'(M)}}

\newcommand{\connsum}{\mathrel\#}
\newcommand{\opConnsum}{\mathop\#}
\newcommand{\opTimes}{\mathop\times}

\begin{document}

\title[The co-rank of the fundamental group]{The co-rank of the fundamental group: the~direct~product, the~first~Betti~number, and~the~topology~of~foliations}
\author{Irina Gelbukh}  
\email{gelbukh@member.ams.org}

\newcommand{\acr}{\newline\indent}
\address{CIC, Instituto Polit\'ecnico Nacional, 07738, Mexico City, MEXICO}

\begin{abstract}
We study $\b$, the co-rank of the fundamental group of a smooth closed connected manifold $M$.
We calculate this value for the direct product of manifolds.
We characterize the set of all possible combinations of $\b$ and the first Betti number $b_1(M)$ 
by explicitly constructing manifolds with any possible combination of $\b$ and $b_1(M)$ in any given dimension.
Finally, we apply our results to the topology of a Morse form foliations. 
In particular, we construct a manifolds $M$ and a Morse form $\w$ on it for any possible combination of $\b$, $b_1(M)$, $m(\w)$, and $c(\w)$, where
$m(\w)$ is the number of minimal components and $c(\w)$ is the 
maximum number of homologically independent compact leaves
of $\w$.
\end{abstract}

\def\sep{, }

\keywords
{
     co-rank\sep 
     inner rank\sep 
     manifold\sep
     fundamental group\sep 
     direct product\sep
     Morse form foliation
}
\makeatletter
\@namedef{subjclassname@2010}{%
  \textup{2010} Mathematics Subject Classification}
\makeatother

\maketitle

\section{Introduction and main results}%

Co-rank of a group $G$, also known as inner rank, %
is the maximum rank of a free homomorphic image of $G$.
In a sense, co-rank is a notion dual to the rank; unlike rank, co-rank is algorithmically computable for finitely presented groups.
This notion has been re-invented various times in different branches of mathematics, and its properties relevant for the corresponding particular task have been studied in different contexts.
The notion of co-rank, called there inner rank, was apparently first mentioned in~\cite{Lyndon59} in the context of solving equations in free groups.
Co-rank of the free product of groups was calculated using geometric~\cite{Jaco72} and algebraic~\cite{Lyndon} methods.

Co-rank is extensively used in geometry, especially in geometry of manifolds, as 
$$\b=\mathop{\mathrm{corank}}(\pi_1(M)),$$ 
the co-rank of the fundamental group $\pi_1(M)$ of a 
manifold $M$. For example, it has been repeatedly shown to coincide with the genus $g$ of a closed oriented surface: $b_1'(M^2_g)=g$~\cite{Lyndon66,Jaco69,Gelb10,Leininger}.

In the theory of 3-manifolds, $\b=c(M)$~\cite{Sikora,Jaco72}, the cut number: 
the largest number $c$ of disjoint two-sided surfaces $N_1 ,\dots,N_c$ that do not separate $M$, i.e., $M \setminus (N_1\cup\dots\cup N_c)$ is connected. 
It is related to quantum invariants of $M^3$ and gives a lower bound on its Heegaard genus~\cite{Gilmer}.
Around 2001, J.~Stallings, A.~Sikora, and T.~Kerler discussed a conjecture that for a closed orientable 3-manifold it holds $\b\ge\frac{b_1(M)}3$, where $b_1(M)$ is the Betti number. This conjecture was later disproved by a number of counterexamples, such as~\cite{Harvey,Leininger}. 
In this paper we, in particular, give a complete characterization of possible pairs $\b$, $b_1(M)$ for any given $\dim M$.

In systolic geometry, every unfree 2-dimensional piecewise flat complex $X$ satisfies the bound $\mathop{S\!R}(X)\le16(b_1'(X)+1)^2$~\cite{KRS}, where $\mathop{S\!R}$ is the optimal systolic ratio.

In the theory of foliations, for the foliation $\F$ defined on $M\setminus\Sing\w$ by a closed 1-form $\w$ with the singular set $\Sing\w$,
it was shown that if $\b\le 1$ and $\codim\Sing\w\ge 3$ with $\Sing\w$ contained in a finite union of submanifolds of $M$, 
then $\F$ has no exceptional leaves~\cite{Levitt87}.
Foliations have numerous applications in physics, such as general relativity~\cite{Chen}, superstring theory~\cite{Bab-Laz,Bab-Laz-Foliated}, etc.

The notion of co-rank of $\pi_1(M)$, the notation $\b$, and the term {\em the first non-commutative Betti number}
were first introduced in~\cite{AL}
to study Morse form  foliations, i.e., foliations defined by a closed 1-form that is locally the differential of a Morse function on a smooth closed manifold $M$.

A Morse form foliation can have compact leaves, compactifiable leaves and minimal components~\cite{Gelb09}. 
In \cite{AL}, it was proved that if $\b\le 2$, then each minimal component of $\F$ is uniquely ergodic. 
On the other hand, if $\b\ge 4$, then there exists a Morse form on $M$ with a minimal component that is not uniquely ergodic.
If the form's rank 
$$
\rk\w>\b,
$$
then the foliation $\F$ has a minimal component~\cite{Levitt90}; here $\rk\w=\rk_\QSet[\w]$, where $[\w]$ is the integration map.

\medskip

Though co-rank is known to be algorithmically computable for finitely presented groups~\cite{Makanin, Razborov}, 
we are not aware of any simple method of finding $\b$ for a given manifold. 
This value is, however, bounded by the isotropy index $h(M)$, which is the maximum rank of a subgroup in $H^1(M,\ZSet)$ with trivial cup-product~\cite{Meln2}. 
Namely, for a smooth closed connected manifold it holds~\cite{Gelb10,Dimca-Pa-Su} 
$$
\b\le h(M),
$$
while for $h(M)$ there are simple estimates via $b_1(M)$ and $b_2(M)$~\cite{Meln3}.

For the connected sum of $n$-manifolds, $n\ge2$, except for non-orientable surfaces, it holds:
\begin{align}
\begin{aligned}
b_1(M_1\connsum M_2)&=b_1(M_1)+b_1(M_2),\\
b_1'(M_1\connsum M_2)&=b_1'(M_1)+b_1'(M_2),
\end{aligned}
\label{eq:b',b(sum)}
\end{align}
which for $\dim M\ge3$ follows from the Mayer--Vietoris sequence and by~\cite[Proposition 6.4]{Lyndon}, respectively.
Also, for the direct product the K\"unneth theorem gives:
\begin{align*}
b_1(M_1\times M_2)=b_1(M_1)+b_1(M_2).
\end{align*}

In this paper, we show that the fourth combination is very different:
\begin{align}
b_1'(M_1\times M_2)=\max\{b_1'(M_1),b_1'(M_2)\}
\label{eq:intro-x}
\end{align}
(Theorem~\ref{theor:b'(dir-prod)}),
which completes the missing piece to allow calculating $\b$ for manifolds that can be represented as connected sums and direct products of simpler manifolds.

We give a complete characterization of
the set of all possible combinations of $\b$ and $b_1(M)$ 
for a given $n=\dim M$. Namely, 
for $b',b\in\ZSet$, there exists a connected smooth closed $n$-manifold $M$ 
with $b'_1(M)=b'$ and $b_1(M)=b$ iff 
\begin{align}
n\ge3\mathrm{:}\quad&b'=b=0\textrm{ or }1\le b'\le b;\label{eq:intro-b'b}\\
n=2\mathrm{:}\quad&0\le b,\ b'=\textstyle{\left[\frac{b+1}2\right]};\notag\\
n=1\mathrm{:}\quad&b'=b=1;\notag\\
n=0\mathrm{:}\quad&b'=b=0;\notag
\end{align}
the manifold can be chosen orientable, except for odd $b$ when $n=2$ (Theorem~\ref{theor:(k,m)-dimM>2}). 
Using~\eqref{eq:b',b(sum)}--\eqref{eq:intro-x}, we explicitly construct such a manifold
(Construction~\ref{con:(k,m)-dimM>2}).

We apply the obtained results 
to estimation of 
the number of minimal components $m(\w)$ 
and 
the maximum number of homologically independent compact leaves $c(\w)$ 
of the foliation $\F$ of a Morse form $\w$ on $M$. 
The smaller $\b$ or $b_1(M)$, the more information we have about $\F$. For example,
\begin{align}
\begin{aligned}
m(\w)+c(\w)&\le\b\ {\textrm{\cite{Gelb10}}},\\
2m(\w)+c(\w)&\le b_1(M)\ {\textrm{\cite{Gelb09}}};
\end{aligned}
\label{eq:intro-m+c}
\end{align}
In particular, if $\b=0$, or, which is the same, $b_1(M)=0$, then all leaves of $\F$ are compact and homologically trivial and $\w=df$ is exact.

Theorem~\ref{theor:(k,m)-dimM>2}, which states that all combinations allowed by~\eqref{eq:intro-b'b} are possible,
implies that the two inequalities are independent. However, in special cases knowing the values of $\b$ and $b_1(M)$, for example, calculated using~\eqref{eq:b',b(sum)}--\eqref{eq:intro-x}, allows choosing one of the two inequalities as stronger. For instance, if $\b\le\frac12b_1(M)$, then the first one is stronger.

Finally, we generalize Theorem~\ref{theor:(k,m)-dimM>2} to a characterization of the set of possible combinations of $\b$, $b_1(M)$, $m(\w)$, and $c(\w)$ (Theorem~\ref{theor:b'bmc}). 
Namely, we use Construction~\ref{con:(k,m)-dimM>2}---the constructive proof of Theorem~\ref{theor:(k,m)-dimM>2}---to show that~\eqref{eq:intro-b'b} and~\eqref{eq:intro-m+c} are the only restrictions on these four values, except for $\b=\frac12b_1(M)$ if $\dim M=2$.

\medskip

The paper is organized as follows.
In Section~\ref{sec:defs}, we give the definitions of the Betti number $b_1$ and the non-commutative Betti number $b_1'$ for groups and manifolds.
In Section~\ref{sec:dir-prod}, we calculate 
$b_1'(M_1\times M_2)$.
In Section~\ref{sec:poss-comb}, we describe the set of possible combinations of $\b$ and  $b_1(M)$ for a given $\dim M$ and, using the results of Section~\ref{sec:dir-prod}, explicitly construct a manifold with any given valid combination of $\b$ and $b_1(M)$.
Finally, in Section~\ref{sec:fol} we 
use the manifold constructed in Section~\ref{sec:poss-comb}
to describe the set of possible combinations of the number of minimal components and the maximum number of homologically independent compact leaves of a Morse form foliation.

\section{Definitions\label{sec:defs}}

For a finitely generated abelian group $G=\ZSet^n\oplus T$, where $T$ is finite, its torsion-free rank, Pr\"ufer rank, or first Betti number, is defined as $b_1(G)=\rk(G/T)=n$. The notion of first Betti number can be extended to any finitely generated group by $b_1(G)=b_1(\HG)=\rk(\HG/\T(\HG))$, where $\HG=G^{ab}=G/[G,G]$ is the abelianization, or the first homology group, of the group $G$, and $\T(\cdot)$ is the torsion subgroup. In other words:

\begin{definition}
The {\it first Betti number} $b_1(G)$ of a finitely generated group $G$ is the maximum rank of a free abelian quotient group of $G$, i.e., the maximum rank of a free abelian group $A$ such that there exists an epimorphism $\phi:G\twoheadrightarrow A$. 
\end{definition}

Consider a connected smooth closed manifold $M$.
The first Betti number of $M$ is the torsion-free rank of its first homology group $H_1(M)$, i.e., of the first homology group of its fundamental group $\pi_1(M)$: $$b_1(M)=b_1(\pi_1(M)).$$ 

A non-commutative analog of the Betti number can be defined as follows.

\begin{definition}
\label{def:corank}
The {\em co-rank\/} $\mathop{\mathrm{corank}}(G)$~\cite{Leininger}, 
{\em inner rank\/} $\mathop{I\kern-1.3pt N}(G)$~\cite{Jaco72} or $\mathop{\mathrm{Ir}}(G)$~\cite{Lyndon}, 
or {\em first non-commutative Betti number} $\crk(G)$~\cite{AL} of a finitely generated group $G$ is the maximum rank of a free quotient group of $G$, i.e., the maximum rank of a free group $F$ such that there exists an epimorphism $\phi:G\twoheadrightarrow F$.
\end{definition}

The notion of co-rank is in a way dual to that of rank, which is the minimum rank of a free group allowing an epimorphism onto $G$. Unlike rank, co-rank is algorithmically computable for finitely presented groups~\cite{Makanin, Razborov}. 

The first non-com\-mu\-ta\-tive Betti number~\cite{AL} of a connected smooth closed manifold $M$ is defined as the co-rank, or inner rank, of its fundamental group: 
$$\b=\crk(\pi_1(M)).$$

Note that a similar definition for higher $\pi_k(M)$ is pointless since they are abelian.

\section{Co-rank of the fundamental group of the direct product\label{sec:dir-prod}}

The Betti number $b_1(M)$ is linear in both connected sum and direct product. 
While the non-commutative Betti number $\b$ is linear in connected sum, its behavior with respect to direct product is very different:

\begin{theorem}\label{theor:b'(dir-prod)}
Let $M_1, M_2$ be connected smooth closed manifolds. Then
$$
b_1'(M_1\times M_2)=\max\{b_1'(M_1),b_1'(M_2)\}.
$$
\end{theorem}

We will divide the proof into a couple of lemmas.

\begin{lemma}\label{lem:F=Z}
Let $G_1,G_2$ be groups. Then any epimorphism 
$$\phi:G_1\times G_2\twoheadrightarrow F$$ 
onto a free group $F\ne\ZSet$ factors through a projection.
\end{lemma}

\begin{proof}
For $F=\{1\}$ the fact is trivial. 
{\em Par abus de langage}, denote $G_1=G_1\times1$ and $G_2=1\times G_2$, subgroups of $G_1\times G_2$.
Suppose both $G_i\not\subseteq\ker\phi$.
Since $G_1\times G_2=\langle G_1,G_2\rangle$ and $[G_1,G_2]=1$ and denoting $F_i=\phi(G_i)$, 
we have a free group $F=\langle F_1,F_2\rangle$ such that $[F_1,F_2]=1$ and by the condition $F_1,F_2\ne1$.

Let $a,b\in F_1$ and $c\in F_2$, $c\ne1$. 
Since $[a,c]=1$, we have $\langle a,c\rangle=\ZSet$ as both free and abelian, 
so $a,c\in\langle x\rangle$ for some $x\in F$, 
and similarly $b,c\in\langle y\rangle$ for some $y\in F$. 
Then $\langle x,y\rangle=\ZSet$ as a two-generated free group with a non-trivial relation $x^k=y^l=c\ne1$, so $x,y\in\langle z\rangle$ for some $z\in F$. We obtain $a,b\in\langle z\rangle$; in particular, $[a,b]=1$. 

Thus $F_1$ is abelian, and similarly $F_2$. Since $[F_1,F_2]=1$, we obtain that the non-trivial $F=\langle F_1,F_2\rangle$ is both free and abelian, thus $F=\ZSet$.
\end{proof}

\begin{remark}%
In fact this holds for any (infinite) quantity of factors:
any epimorphism $\phi:\opTimes_{\alpha\in I}G_\al\twoheadrightarrow F$ onto a free group $F\ne\ZSet$ factors through the projection onto one of\/ $G_\al$.
\end{remark}

\begin{lemma}\label{lem:co-rank(dir-prod)}
Let $G_1, G_2$ be finitely generated groups. Then
for the co-rank of the direct product, 
$$\crk(G_1\times G_2)=\max\{\crk(G_1),\crk(G_2)\}.$$
\end{lemma}

\begin{proof}%
Denote $G=G_1\times G_2$ and $m=\max\{\crk(G_1),\linebreak[0]\crk(G_2)\}$. 

Let us show that $\crk(G)\ge m$. Without loss of generality assume $m=\crk(G_1)$. Consider an epimorphism $\phi:G_1\twoheadrightarrow F$ onto a free group, $\rk F=m$. 
Then $\psi:G\twoheadrightarrow F$ such that $\psi(G_2)=1$ and $\psi|_{G_1}=\phi$ is an epimorphism, so $\crk(G)\ge m$.

Let us now show that $m\ge\crk(G)$. Consider an epimorphism $\phi:G\twoheadrightarrow F$ onto a free group, $\rk F=\crk(G)$. 
If $\phi(G_2)=1$, then $\phi|_{G_1}$ is an epimorphism and thus $m\ge\crk(G_1)\ge\rk F=\crk(G)$, and similarly if $\phi(G_1)=1$.

Otherwise $\rk\im\phi|_{G_i}\ge1$, so $m\ge\crk(G_i)\ge1=\crk(G)$ by Lemma~\ref{lem:F=Z}.
\end{proof}

\begin{proof}[of Theorem~\ref{theor:b'(dir-prod)}]
For smooth connected manifolds $M_i$, we have
$$
\pi_1(M_1\times M_2)=\pi_1(M_1)\times\pi_1(M_2),
$$
and the desired fact is given by Lemma~\ref{lem:co-rank(dir-prod)}.
\end{proof}

\begin{example}%
For a torus $T^n=\opTimes_{i=1}^nS^1$, we have
$b'_1(T^n)=b'_1(S^1)=1$. 
Since $\pi_1(T^n)\ne\{0\}$ is free abelian, this also follows from 
Definition~\ref{def:corank}.%

For the Kodaira-Thurston manifold $KT^4=H^3\times S^1$, we have $b'_1(KT^4)=1$ since for the Heisenberg nil manifold, $b'_1(H^3)=1$. This also follows from the fact that $KT^4$ itself is a nil manifold.
\end{example}

In Section~\ref{sec:poss-comb}, we will use Theorem~\ref{theor:b'(dir-prod)} to explicitly construct a manifold with arbitrary given $b'(M)$ and $b(M)$.

\section{Relation between $\b$ and $b_1(M)$\label{sec:poss-comb}}

As an 
application of Theorem~\ref{theor:b'(dir-prod)}, in this section we show 
that there are no non-obvious relations between $b_1'(M)$ and $b_1(M)$, and explicitly construct a manifold with any given valid pair of $b'_1(M)$ and $b_1(M)$.

For a sphere and for low-dimensional manifolds, such as closed orientable surface $M^2_g=\opConnsum^gT^2$ and closed non-orientable surface $N^2_h=\opConnsum^h\RSet P^2$, $h\ge1$, the values of $b_1'(M)$ and $b_1(M)$ are obvious or well known:
\begin{align}
\begin{aligned}
&\textrm{point:}&&b'_1(*)=0,&&b_1(*)=0;\\
&\textrm{circle:}&&b'_1(S^1)=1,&&b_1(S^1)=1;\\
&\textrm{sphere,}\ n\ge2:&&b'_1(S^n)=0,&&b_1(S^n)=0;\\
&\textrm{orientable surface:}&&b'_1(M^2_g)=g,&&b_1(M^2_g)=2g;\\
&\textrm{non-orientable surface:}&&b'_1(N^2_h)={\textstyle\left[\frac h2\right]},&&b_1(N^2_h)=h-1.
\end{aligned}
\label{eq:b(M^2_g)}
\end{align}
The value of 
$b'_1(M^2_g)$ was calculated in~\cite{Gelb10} and~\cite[Lemma 2.1]{Leininger}.
It can be also obtained 
as the cut-number~\cite[Theorem 2.1]{Jaco72}, which for a surface is the number of handles, since each non-separating two-sided circle identifies the edges of two holes. In particular, $b'_1(N^2_h)$ is a sphere with $\left[\frac h2\right]$ inverted handles plus a M\"obius strip for odd $h$.

In general,
\begin{align}
0\le b'_1(M)\le b_1(M)
\label{eq:0<b'M<bM}
\end{align}
with
\begin{align}
b_1'(M)=0\textrm{ iff\/ }b_1(M)=0,\label{eq:iff}
\end{align}
and thus
\begin{align}
b_1(M)=1\textrm{ implies\/ }b_1'(M)=1.\label{eq:implies}
\end{align}

Indeed, since for a free group $F$ it holds $\rk F^{ab}=\rk F$ and a group epimorphism $G\twoheadrightarrow F$ induces an epimorphism $G\twoheadrightarrow F^{ab}$, it  holds
\begin{align*}%
0\le b'_1(G)\le b_1(G)
\end{align*}
and since $\ZSet$ is both free and free abelian, $\crk(G)=0$ iff $b_1(G)=0$.

There are no relations between $b_1'(M)$ and $b_1(M)$ other than~\eqref{eq:b(M^2_g)}--\eqref{eq:iff}:

\begin{theorem}\label{theor:(k,m)-dimM>2}
Let $b',b,n\in\ZSet$. There exists a connected smooth closed $n$-manifold $M$
with $b'_1(M)=b'$ and $b_1(M)=b$ iff 
\begin{align}
n\ge3\mathrm{:}\quad&b'=b=0\textrm{ or }1\le b'\le b;\label{eq:1<b'<b}\\
n=2\mathrm{:}\quad&0\le b,\ b'=\textstyle{\left\lfloor\frac{b+1}2\right\rfloor};\notag\\
n=1\mathrm{:}\quad&b'=b=1;\notag\\
n=0\mathrm{:}\quad&b'=b=0;\notag
\end{align}
the manifold can be chosen orientable iff $n\ne2$ or $b$ is even. 
\end{theorem}

\begin{proof}
For $n\le2$ and for $b=0$ the facts are given in~\eqref{eq:b(M^2_g)},
so let $n\ge 3$ and $b\ge1$.

For $n\ge4$, every finitely presented group is the fundamental group of a connected smooth closed orientable $n$-manifold $M$, while
by~\cite[Theorem 3]{Gelb17} there exists such a group $G$ with $\crk(G)=b'$ and $b_1(G)=b$, which proves the result for $n\ge4$.

Finally, let $n=3$. 
For any given $b\ge1$, Harvey~\cite[Theorem 3.1]{Harvey} constructed a smooth closed orientable hyperbolic 3-manifold $H_b$ with the largest possible gap between $b_1'$ and $b_1$: 
\begin{align}
b_1'(H_b)=1,\ b_1(H_b)=b.
\label{eq:H_b}
\end{align}
For $1\le b'\le b$, choose $k_i\ge1$ such that $\sum_{i=1}^{b'} k_i=b$.
By~\eqref{eq:b',b(sum)}, for 
\begin{align}
M=\opConnsum_{i=1}^{b'} H_{k_i}
\label{eq:H31b}
\end{align}
we have $b_1'(M)=b',\ b_1(M)=b$.
\end{proof}

Using Theorem~\ref{theor:b'(dir-prod)}, we can generalize any specific example with given $\crk(M)$ and $b_1(M)$ to higher dimensions, 
as well as to increase the gap between $b_1'(M)$ and $b_1(M)$:

\begin{lemma}\label{lem:xS^n}
Let $b_1'(M^n)\ne0$. Then for $M^{n+k}=M^n\times S^k$ it holds
\begin{align*}
k=1:&\qquad b'_1(M^{n+k})=b'_1(M^n),&&b_1(M^{n+k})=b_1(M^n)+1;\\
k\ge2:&\qquad b'_1(M^{n+k})=b'_1(M^n),&&b_1(M^{n+k})=b_1(M^n).
\end{align*}
\end{lemma}

This allows us to explicitly construct a manifold with given $b_1'(M)$ and $b_1(M)$ of any given dimension, thus giving a simple constructive proof
of Theorem~\ref{theor:(k,m)-dimM>2} for $\dim M\ge3$:

\begin{construction}\label{con:(k,m)-dimM>2}
For any given $b',b,n\in\ZSet$ such that $b'=b=0$ or $1\le b'\le b$, and $n\ge3$,
the following connected smooth closed oriented $n$-manifold $H^n_{b',b}$ has
\begin{align*}
b_1'(H^n_{b',b})=b',\ b_1(H^n_{b',b})=b.
\end{align*}
For $b'=b=1$, consider
\begin{align}
H^n_{1,1}=S^1\times S^{n-1}
\label{eq:Hn11}
\end{align}
and for $b\ge2$, generalize~\eqref{eq:H_b} 
to higher dimensions using Lemma~\ref{lem:xS^n}:
\begin{align}
H^n_{1,b}=
\begin{cases}
H_b&$for $n=3,\\
H_{b-1}\times S^1&$for $n=4,\\
H_b\times S^{n-3}&$for $n\ge5.
\end{cases}
\label{eq:Hn1b}
\end{align}
Finally, as in~\eqref{eq:H31b},
choose $k_i\ge1$ such that 
\begin{align}
\sum_{i=1}^{b'} k_i=b
\label{eq:sum-m_i=b}
\end{align}
and take
\begin{align}
H^n_{b',b}=\opConnsum_{i=1}^{b'} H^n_{1,k_i}.
\label{eq:Hnb'b}
\end{align}
\end{construction}

By Theorem~\ref{theor:(k,m)-dimM>2},
in~\eqref{eq:0<b'M<bM} both the lower bound (except for $n=1$) and the upper bound (except for surfaces other than $S^2$, $\RSet P^2$, and the Klein bottle) are exact for any given $n$. Both conditions~\eqref{eq:iff} are impossible for $n=1$ and both conditions~\eqref{eq:implies} are impossible for $n=0$. 

In particular, the lower bound in~\eqref{eq:0<b'M<bM} is achieved on $S^n$, while~\eqref{eq:Hn11} and~\eqref{eq:Hn1b} provide the lower bound in the inequality in~\eqref{eq:1<b'<b}.
The upper bound $\b=b_1(M)$ for $n\ge3$ is provided by~\eqref{eq:Hnb'b} with $b'=b$:
$$H^n_{b,b}=\opConnsum_{i=1}^b(S^1\times S^{n-1}).$$
In general, 
$\b=b_1(M)$ iff some (and thus any) epimorphism $$\pi_1(M)\twoheadrightarrow H_1(M)/\T(H_1(M))$$
factors through a free group; $\T(\cdot)$ is the torsion subgroup:

\begin{proposition}%
For any group $G$, the following conditions are equivalent:
\begin{enumerate}
\renewcommand{\labelenumi}{$\mathrm{\theenumi}$}
\renewcommand{\theenumi}{(\roman{enumi})}
\item\label{lem:co-rank=b:=} $\crk(G)=b_1(G)$, 
\item\label{lem:co-rank=b:E} there exists an epimorphism $$h:G\twoheadrightarrow \ZSet^{b_1(G)}=\HG/T$$ that factors through a free group; $T\subset\HG$ is the torsion subgroup,
\item\label{lem:co-rank=b:A} any such epimorphism factors through a free group.
\end{enumerate}
\end{proposition}

\begin{proof}
\ref{lem:co-rank=b:=}\,$\Rightarrow$\,\ref{lem:co-rank=b:E}:
Let $\phi:G\twoheadrightarrow F$ be the epimorphism from the definition of $\crk(G)$ and $\psi:F\twoheadrightarrow F/[F,F]=\ZSet^{\crk(G)}=\ZSet^{b_1(G)}$ be the natural epimorphism; then $h=\psi\circ\phi$ has the desired properties.

\ref{lem:co-rank=b:E}\,$\Rightarrow$\,\ref{lem:co-rank=b:=}:
Let $G\twoheadrightarrow F\twoheadrightarrow\HG/T$ be a factorization of $h$ through a free group $F$. 
Then $b_1(G)=\rk (\HG/T)\le\rk F\le\crk(G)$. By~\cite[Theorem 3]{Gelb17}, $\crk(G)\le b_1(G)$, so we obtain $\crk(G)=b_1(G)$.

\ref{lem:co-rank=b:E}\,$\Rightarrow$\,\ref{lem:co-rank=b:A}: 
For any epimorphisms $\phi,h:G\twoheadrightarrow H$ there exists an automorphism $\psi:H\twoheadrightarrow H$ such that $\phi=\psi\circ h$.

\ref{lem:co-rank=b:A}\,$\Rightarrow$\,\ref{lem:co-rank=b:E}:
The composition of natural epimorphisms $$G\twoheadrightarrow\HG=G/[G,G]\twoheadrightarrow\HG/T=\ZSet^{b_1(G)}$$ is an epimorphism.
\end{proof}

\section{Application to foliation topology\label{sec:fol}}

The gap between $\b$ and $b_1(M)$ plays a role in foliation topology.

\subsection{Useful facts about Morse form foliations}%

Consider a connected closed oriented $n$-manifold $M$ with a Morse form $\w$, i.e., a closed 1-form
with Morse singularities---locally the differential of a Morse
function. The set of its singularities $\Sing\w$ is finite.
This form defines a foliation $\F$ on $M\setminus\Sing\w$. 

Its leaves $\g$ can be classified into compact, compactifiable ($\g\cup\Sing\w$ is compact), and non-compactifiable.
The set covered by all non-compactifiable leaves is open and has a finite number $m(\w)$ of connected components, called {\em minimal components}~\cite{AL}. Each non-compactifiable leaf is dense in its minimal component~\cite{Im}. %
A foliation is called {\em minimal} if all its leaves are non-compactifiable, i.e., the whole $M\setminus\Sing\w$ is one minimal component.

Any compact leaf has a cylindrical neighborhood consisting of leaves that are diffeomorphic and homotopically equivalent to it~\cite{Gelb14}.
Denote by $H_\w\subseteq H_{n-1}(M)$ the group generated by the homology classes of all compact leaves of $\F$. Since $M$ is closed and oriented, $H_{n-1}(M)$ is finitely generated and free; therefore so is $H_\w$. By~\cite[Theorem 4]{Gelb08}, in $H_\w$ there exists a basis consisting of homology classes of leaves, i.e., $\F$ has exactly $c(\w)=\rk H_\w$ homologically independent compact leaves. 

\begin{lemma}\label{lem:m+m,c+c}
Let $\w_1,w_2$ be Morse forms defined on smooth closed oriented manifolds $M_1,M_2$, respectively.
Then on $M=M_1\connsum M_2$ there exists a Morse form $\w$ with $m(\w)=m(\w_1)+m(\w_2)$ and\/ $c(\w)=c(\w_1)+c(\w_2)$.
\end{lemma}

\begin{proof}
Consider a form $\w$ constructed as shown in \figurename~\ref{fig:connsum}. 
It coincides with $\w_i$ outside a small area where $M_i$ are glued together.
We assume that each $\w_i$ was locally distorted either in a minimal component or in a cylindrical neighborhood covered by homologous compact leaves.
In the former case, since nearby leaves are dense on either side of the affected leaf, the distortion does not change the number of minimal components.
In the latter case, even though the distortion ``destroys'' one compact leaf, the nearby leaves contribute the same value to $H_{\w_i}$. 
In either case, the new compact leaves introduced in the process are homologically trivial.

\begin{figure}[t]
\centering
\setlength\unitlength{0.34ex}
\vspace{0\unitlength}%
\includegraphics[width=100\unitlength]{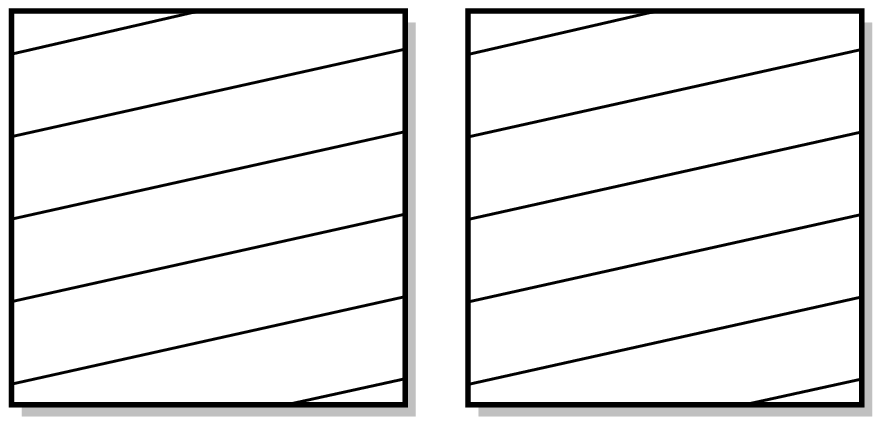}%
\begin{picture}(0,0)(100,0)
\put(50,-7){\makebox[0pt][c]{({\it a})}}
\end{picture}
\hspace{3ex}
\includegraphics[width=100\unitlength]{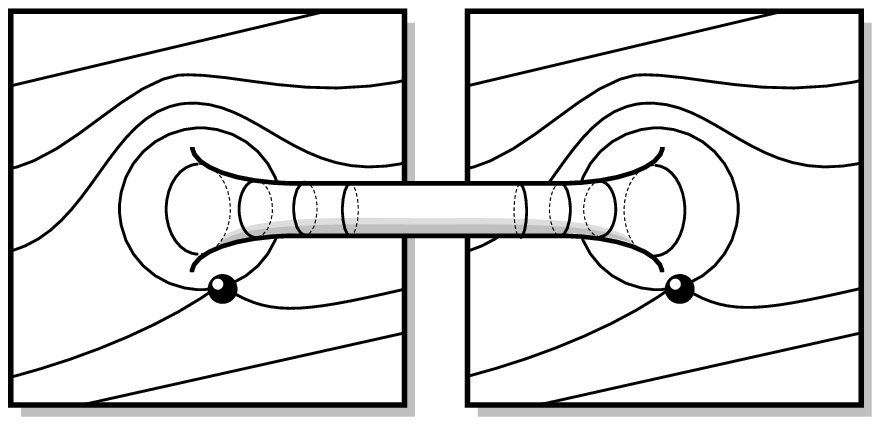}%
\begin{picture}(0,0)(100,0)
\put(50,-7){\makebox[0pt][c]{({\it b})}}
\end{picture}%
\vspace{7\unitlength}%
\caption{\label{fig:connsum}
Construction of the form on the connected sum. 
(a) Manifolds $M_1$ and $M_2$ with Morse forms $\w_1$, $\w_2$, respectively.
(b) A center and a conic singularity are locally added to each form, and the manifolds are glued together by spheres around the removed centers.
The new form $\w$ on $M_1\connsum M_2$ has two additional conic singularities.}
\end{figure}

Since the two sides are separated by compact leaves, each minimal component of $\w$ lies either in $M_1$ or in $M_2$, and thus $m(\w)=m(\w_1)+m(\w_2)$.
Similarly, 
homologically non-trivial leaves of $\w$ are homologous to either leaves of $\w_1$ or leaves of $\w_2$; 
in particular, $H_\w=H_{\w_1}\oplus H_{\w_2}$ and thus $c(\w)=c(\w_1)+c(\w_2)$.
\end{proof}

\subsection{Relation between $\b$, $b_1(M)$, $m(\w)$, and\/ $c(\w)$}%

Let $\w$ be a Morse form on a smooth closed orientable manifold $M$, $\dim M\ge2$, defining a foliation with exactly $c(\w)$ homologically independent compact leaves and $m(\w)$ minimal components. 
The following inequalities
have been proved independently: 
\begin{align}
m(\w)+c(\w)&\le\b\ {\textrm{\cite{Gelb10}}},\label{eq:c+m}\\
2m(\w)+c(\w)&\le b_1(M)\ {\textrm{\cite{Gelb09}}}.\label{eq:c+2m}
\end{align}

For some manifolds,~\eqref{eq:b',b(sum)} and Theorem~\ref{theor:b'(dir-prod)} allow direct calculation of $\b$. This may allow one to choose between~\eqref{eq:c+m} and~\eqref{eq:c+2m}. 
Namely, denoting $b'=\b$, $b=b_1(M)$, unless $b'=b=0$ we have:
\begin{enumerate}
\renewcommand{\labelenumi}{$\mathrm{\theenumi}$}
\renewcommand{\theenumi}{(\roman{enumi})}
\item If\/ $b'\le\frac12b$, then~\eqref{eq:c+m} is stronger;
\item\label{case:independent}
      If\/ $\frac12b<b'<b$, then they are independent;
\item If\/ $b'=b$, then~\eqref{eq:c+2m} is stronger. 
\end{enumerate}

In particular,~\eqref{eq:c+m} is always stronger for $\dim M=2$. 
However, for any $\dim M\ge3$, by Theorem~\ref{theor:(k,m)-dimM>2} all three cases are possible; in particular, there exist manifolds for which the two estimates are independent.
By~\eqref{eq:0<b'M<bM}, the case $b'>b$ is impossible.

More specifically, in the case~\ref{case:independent}, if seen as conditions on $c(\w)$ under given $m(\w)$ and vice versa, 
\begin{itemize}
\renewcommand{\labelitemi}{--}
\itemsep0em
\item \eqref{eq:c+m} is stronger when $m(\w)<b-b'$ or when $c(\w)>2b'-b$;
\item \eqref{eq:c+2m} is stronger when $m(\w)> b-b'$ or when $c(\w)<2b'-b$,
\end{itemize}
and they are equivalent in case of equalities.

We can generalize 
Theorem~\ref{theor:(k,m)-dimM>2} to observe that
there are no relations between $m(\w)$, $c(\w)$, $b'_1(M)$, and $b_1(M)$ other than those given by~\eqref{eq:c+m}, \eqref{eq:c+2m}, \eqref{eq:0<b'M<bM}, \eqref{eq:iff}, and, for an orientable surface,~\eqref{eq:b(M^2_g)}:

\begin{theorem}\label{theor:b'bmc}
Let\/ $n,m,c,b',b\in\NSet^0$. There exists a smooth closed connected oriented $n$-manifold\/ $M$ with $\b=b'$ and $b_1(M)=b$, 
and a Morse form foliation $\F$ on it with $m$ minimal components and exactly $c$ homologically independent compact leaves, iff 
\begin{alignat}{3}
n=2\mathrm{:}\quad
&\text{\eqref{eq:b(M^2_g)} for $M^2_g$\upshape{:}}\quad &&0\le b=2b',\notag\\
&\text{\eqref{eq:c+m}\upshape{:}}                                                                                               &&0\le m+c\le b',\notag\\
n\ge3\mathrm{:}\quad
&\text{\eqref{eq:0<b'M<bM}, \eqref{eq:iff}\upshape{:}}          &&b'=b=0\textrm{ or }1\le b'\le b,\label{eq:Th-0<b'M<bM}\\
&\text{\eqref{eq:c+m}\upshape{:}}                                                                                               &&0\le m+c\le b',\label{eq:Th-c+m}\\
&\text{\eqref{eq:c+2m}\upshape{:}}                                                                                              &&0\le 2m+c\le b.\label{eq:Th-c+2m}
\end{alignat}
\end{theorem}

Apart from a trivial foliation on $S^n$, the proof is given by the following constructions. 

\begin{lemma}
On $S^k\times S^1$, $k\ge1$, there exists Morse form foliation with $m(\w)=0$, $c(\w)=0$.
\end{lemma}

\begin{proof} 
The corresponding foliations for $k=1$ and $k\ge2$ are shown in Figures~\ref{fig:ex1} and~\ref{fig:ex2}, respectively.
\end{proof}

\begin{figure}[t]
\centering
\setlength\unitlength{0.2ex}
\vspace{7\unitlength}
\includegraphics[width=100\unitlength]{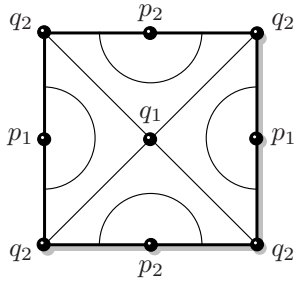}%
\begin{picture}(0,0)(100,0)
\put(-1,100){\makebox[0pt][r]{$q_2$}}
\put(-1,50){\makebox[0pt][r]{$p_1$}}
\put(-1,0){\makebox[0pt][r]{$q_2$}}
\put(102,100){$q_2$}
\put(102,50){$p_1$}
\put(102,0){$q_2$}
\put(50,104){\makebox[0pt][c]{$p_2$}}
\put(50,60){\makebox[0pt][c]{$q_1$}}
\put(50,-7){\makebox[0pt][c]{$p_2$}}
\end{picture}%
\vspace{7\unitlength}%
\caption{\label{fig:ex1}
Torus $T^2$ foliated with $m(\w)=0$, $c(\w)=0$. 
The square is self-glued by the sides to form a torus.
The foliation, shown in thin lines, has two centers $p_1$, $p_2$ and two conic singularities $q_1$, $q_2$.}
\end{figure}

\begin{figure}[t]
\centering
\setlength\unitlength{0.36ex}
\vspace{0\unitlength}
\includegraphics[width=100\unitlength]{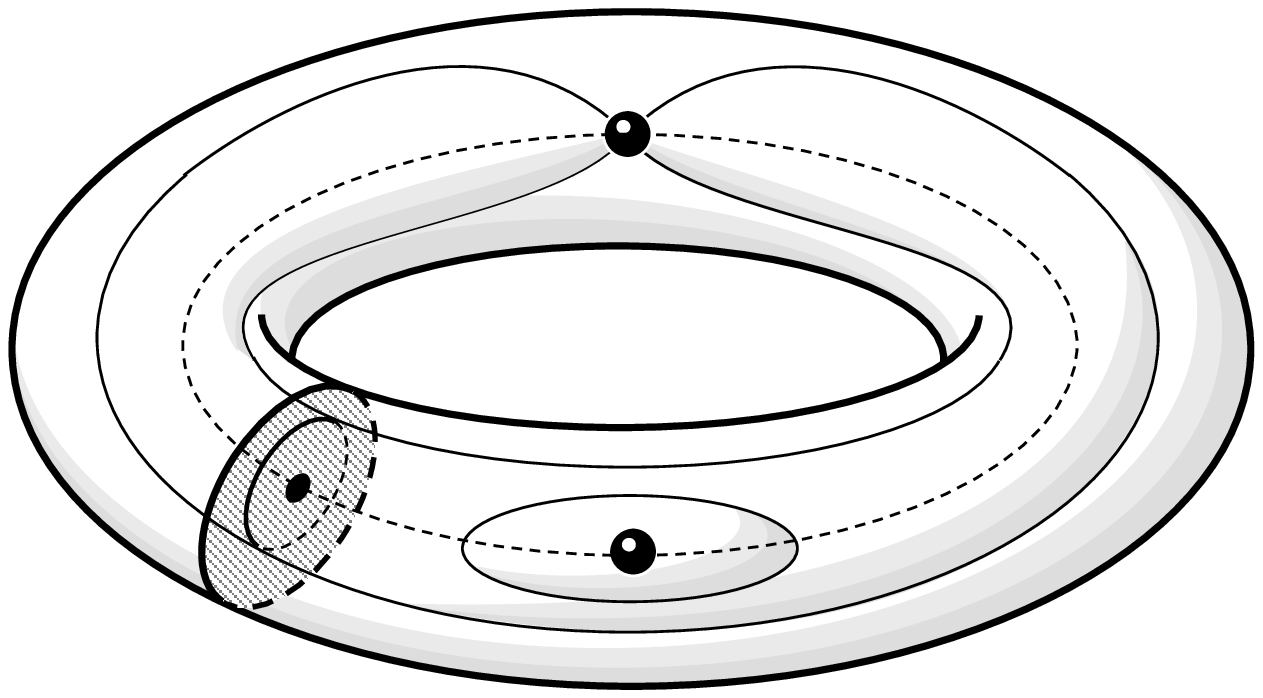}%
\begin{picture}(0,0)(100,0)
\put(6,1){$D^k_1$}
\put(15,37){$S^1$}
\end{picture}
\hspace{3ex}
\includegraphics[width=100\unitlength]{Example-2}%
\begin{picture}(0,0)(100,0)
\put(6,1){$D^k_2$}
\put(15,37){$S^1$}
\end{picture}%
\vspace{0\unitlength}%
\caption{\label{fig:ex2}
Manifold $S^k\times S^1$, $k\ge2$, foliated with $m(\w)=0$, $c(\w)=0$. 
Shown are two manifolds $D^k_i\times S^1$ (solid tori for $k=2$), where $D^k_i$ are disks and $S^1$ is shown in dashed line
They are glued together by the boundary, so that $D^k_1\cup D^k_2=S^k$.
The foliation, shown in thin lines, has two centers and two conic singularities, its compact leaves being either spheres $S^k$ or $S^{k-1}\times S^1$ (tori for $k=2$).}
\end{figure}

\begin{lemma}\label{lem:M2g}
For $0\le m+c\le g$, on $M=M^2_g$ there exists a Morse form foliation with $m(\w)=m$ and\/ $c(\w)=c$.
\end{lemma}

\begin{proof} 
As shown in \figurename~\ref{fig:M2g}, construct $M^2_g$ as the connected sum of
\begin{itemize}
\renewcommand{\labelitemi}{--}
\itemsep0em
\item $m$ tori with an irrational winding;
\item $c$ tori $T^2$ with a compact non-singular foliation;
\item $g-(c+m)$ tori with a foliation shown in \figurename~\ref{fig:ex1};
\end{itemize}
with the forms glued together as shown in \figurename~\ref{fig:connsum}. 
Lemma~\ref{lem:m+m,c+c} completes the proof.
\end{proof} 

\newcommand{\UB}[2]{\ensuremath{\underbrace{#1\rule[-5.6ex]{0ex}{0ex}}_{\textstyle #2\rule[1.3ex]{0ex}{0ex}}}}
\newcommand\TwoPictures[1]{\vcenter{\hbox{\includegraphics[width=10ex]{#1}}}\cdots\vcenter{\hbox{\includegraphics[width=10ex]{#1}}}}

\begin{figure}[t]
  \centering
  \UB{
  \UB{\TwoPictures{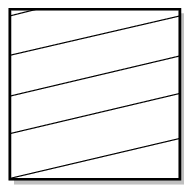}}{m}
        \hspace{2ex}
        \UB{\TwoPictures{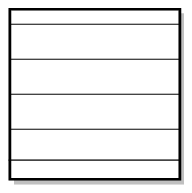}}{c\rule[-0.6ex]{0ex}{0ex}}
        \hspace{2ex}
        \TwoPictures{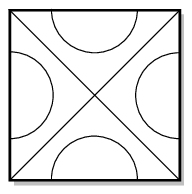}
        }g
  \caption{\label{fig:M2g}
  Construction of $M^2_g$ in the proof of Lemma~\ref{lem:M2g} as the connected sum of tori with different foliations.}
\end{figure}

This easily generalizes to higher dimensions:

\begin{lemma}
For $n\ge3$ and\/ $m,c,b',b\in\NSet^0$ satisfying~\eqref{eq:Th-0<b'M<bM}--\eqref{eq:Th-c+2m}, there exists an $n$-manifold\/ $H^n_{b',b}$ from Construction~\ref{con:(k,m)-dimM>2} and a Morse form on it with $c(\w)=c$ and\/ $m(\w)=m$.
\end{lemma}

\begin{proof}
The construction is very similar to that of Lemma~\ref{lem:M2g}.
Observe that in~\eqref{eq:sum-m_i=b}, we can choose at most $b-b'$ summands $k_i\ge2$. 

If $b-b'\ge m$ (which by~\eqref{eq:Th-c+m} and~\eqref{eq:Th-c+2m} is always the case if $b'\le\frac12b$), 
then in~\eqref{eq:sum-m_i=b} choose $m$ summands $k_i\ge2$, which by~\eqref{eq:Th-c+m} leaves at least $c$ 
summands $k_i=1$:
\begin{align*}
H^n_{b',b}=\left(\opConnsum_{i=1}^m H^n_{1,k_i}\right)\connsum\left(\opConnsum_{i=1}^c H^n_{1,1}\right)\connsum\left(\opConnsum_{\vphantom{i=1}}  H^n_{1,1}\right).
\end{align*}
Otherwise, in~\eqref{eq:sum-m_i=b} choose $b-b'$ summands $k_i=2$, which by~\eqref{eq:Th-c+2m} leaves 
$$b'-(b-b') = b - 2(b-b') \ge (2m+c) - 2(b-b') = 2(m-(b-b'))+c$$
summands $k_i=1$:
\begin{align*}
H^n_{b',b}=
                \underbrace{
                \left(\opConnsum_{i=1}^{b-b'} H^n_{1,2}\right)
        \connsum\left(\opConnsum_{i=1}^{m-(b-b')}\left(H^n_{1,1}\connsum H^n_{1,1}\right)\right)
        }_{m\text{ manifolds}}
        \connsum\left(\opConnsum_{i=1}^c H^n_{1,1}\right)
        \connsum\left(\opConnsum_{\vphantom{i=1}} H^n_{1,1}\right).
\end{align*}

As in Lemma~\ref{lem:M2g}, the foliations on the manifolds are chosen as follows:
\begin{itemize}
\renewcommand{\labelitemi}{--}
\itemsep0em
\item On $m$ manifolds $H^n_{1,k_i}$, $k_i\ge2$, and $H^n_{1,1}\connsum H^n_{1,1}$ there exists a Morse form $\w$ with a minimal foliation~\cite[Theorem~1]{AL}; 
in particular, $m(\w)=1$ and $c(\w)=0$;
\item On $c$ manifolds $H^n_{1,1}=S^1\times S^{n-1}$, choose a compact non-singular foliation along $S^1$ with leaf $S^{n-1}$;
\item On the rest of $H^n_{1,1}$, choose a foliation shown in \figurename~\ref{fig:ex2};
\end{itemize}
Lemma~\ref{lem:m+m,c+c} completes the proof.
\end{proof}

\bibliographystyle{amsplain-maslo}
\bibliography{Gelbukh}

\end{document}